\documentclass[envcountsame]{llncs}

\usepackage[utf8]{inputenc}
\usepackage[T1]{fontenc}
\usepackage{amsmath,amssymb,amsfonts}

\begin{document}
\mainmatter
\title{Cardinal-Recognizing Infinite Time Turing Machines}
\author{Miha E. Habič}
\institute{The Graduate Center of the City University of New York, Mathematics Program,\\
365 Fifth Avenue, New York, NY 10016, USA\\
\email{mhabic@gc.cuny.edu}}
\maketitle

\begin{abstract}
We introduce a model of infinitary computation which enhances the infinite time Turing machine
model slightly but in a natural way by giving the machines the capability of detecting
cardinal stages of computation.
The computational strength with respect to ITTMs is determined to be precisely that of 
the strong halting problem
and the nature of the new characteristic ordinals (clockable, writable, etc.) is explored.
\end{abstract}

Various notions of infinitary computability have now been studied for several decades. The concept that we can somehow utilize infinity
to accommodate our computations is at the same time both appealing and dangerous; appealing, since we are often in a position where we could answer some question if only we could
look at the output of some algorithm after an infinite amount of steps, and dangerous, since this sort of greediness must inevitably lead to disappointment when we suddenly
reach the limits of our model. At that point we must decide whether to push on and strengthen our model in some way or to abandon it in favour of some (apparently) alternative model.
But if we do not wish to abandon our original idea, how to strengthen it in such a way that it remains both interesting and intuitive?

The behaviour of infinite time Turing machines (ITTMs), first introduced in \cite{hamkinslewis-ittm:2000}, has by now been extensively explored and various characteristics have been determined. While
we are far from reaching full understanding of the model, we nevertheless already feel the urge to generalize further. Perhaps the most direct generalization are the ordinal Turing
machines of \cite{koepke-otm:2005}, where both the machine tape and running time are allowed to range into the transfinite. But perhaps this modification seems too strong; with it
all constructible sets of ordinals are computable. We would be satisfied with the minimal nontrivial expansion of the ITTM model, i.e.\ something which computes the appropriate
halting problem but no more. Of course, we can easily achieve this goal within the ITTM framework by considering oracle computations, but this somehow doesn't seem satisfactory.
We intend to give what we feel is a more natural solution to this problem but which falls short of the `omnipotence' of ordinal Turing machines.
%

We assume some familiarity with the concepts and notation related to ITTMs and computability
theory in general. In particular, we fix at the outset some uniform way of coding programs,
whole machine configurations and countable ordinals as reals (i.e.\ infinite binary sequences).
Given a code \(p\) for a program, \(\varphi_p(x)\) denotes the computation of \(p\) with \(x\)
as input, while \(\varphi_p(x)\downarrow\) means that this computation halts. By the output of
a computation stabilizing, we mean that from some time onward the contents of the output tape
do not change during that computation (but the computation itself need not halt or even be
aware of the stabilization). An ordinal is \emph{clockable} if it is the halting time of some
ITTM computation with empty input; the supremum of the clockable ordinals is denoted
\(\gamma\). A real is \emph{writable} if it is the output of some halting ITTM computation with
empty input, \emph{eventually writable} if it is the output of a stabilized ITTM computation
with empty input, and \emph{accidentally writable} if it appears on any tape at any time
during an ITTM computation with empty input. An ordinal is (eventually/accidentally) writable
if the real coding it is such. The suprema of the writable/eventually writable/accidentally
writable ordinals are denoted \(\lambda,\zeta\) and \(\Sigma\), respectively.

\section{The Model and Its Computational Power}

We build on the standard ITTM framework in which the machines have three tapes and cell values at limit stages are calculated according to the \(\limsup\) rule.
Our proposed model has the same hardware and behaviour, with one exception: there is a special state, called the \textit{cardinal} state, which is used instead of the \textit{limit}
state at cardinal stages of the computation. To be precise, if \(\kappa\) is an uncountable\footnote{We restrict to uncountable cardinals mainly to ensure that verbatim copies of ITTM
programs work as expected. There is no difference in power between this convention and using the \textit{cardinal} state at all cardinal stages of computation.} 	
cardinal, the configuration of our machine at stage \(\kappa\) is as follows: the head is on the first cell of the tape, the machine is in the \textit{cardinal} state and the cell
values are the \(\limsup\) of the previous values. The machine handles non-cardinal limit ordinal stages like an ordinary ITTM. We call these machines \emph{cardinal-recognizing
infinite time Turing machines} (CRITTMs). We also define CRITTM-computable functions on 
Cantor space \(2^\omega\) and CRITTM-(semi)decidable subsets of \(2^\omega\) analogously
to the ITTM case.

Our first task should be to verify that our proposed machines actually possess the computing power we desired of them. Recall the characteristic undecidable problems of
ITTM computation:
\begin{itemize}
\item the weak halting problem \(0^\triangledown=\{p;\varphi_p(0)\downarrow\}\),
\item the strong halting problem \(0^\blacktriangledown=\{(p,x);\varphi_p(x)\downarrow\}\),
\item the stabilization problem \(S=\{(p,x);\text{ the output of \(\varphi_p(x)\) stabilizes}\}\).
\end{itemize}

\begin{proposition}
\label{prop:halting decidable}
The sets \(0^\triangledown,0^\blacktriangledown\) and \(S\) are CRITTM-decidable.
\end{proposition}

\begin{proof}
The decidability of \(0^\triangledown\) is clearly reducible to the decidability of \(0^\blacktriangledown\). Similarly, the decidability of \(0^\blacktriangledown\) is reducible
to the decidability of \(S\): given a program \(p\) and an input \(x\), construct a new program \(p'\) which acts like \(p\) but also flashes a designated cell (called a \emph{flag}) after
completing each instruction of \(p\); the program \(p\) halts on \(x\) iff \(p'\) 
stabilizes on \(x\).

It therefore remains to show that \(S\) is CRITTM-decidable. Consider the following algorithm: given a pair \((p,x)\), simulate \(\varphi_p(x)\) and flash a flag each time the
simulated output changes. When a \textit{cardinal} state is attained, output `no' if the flag is showing 1 and `yes' if it is showing 0. This algorithm decides
\(S\). Indeed, recall that every ITTM computation either halts or begins repeating at some countable time and so the question of stabilization is completely settled by (or before)
time \(\omega_1\) when the \textit{cardinal} state is first reached.
\qed
\end{proof}


In the above proof we made use of the fact that every ITTM computation halts or begins
repeating before time \(\omega_1\). An analogous property also holds for CRITTM computations.
To state this correctly we must, as in the ITTM case, clarify what we mean by an infinite computation repeating. In the
ITTM model this is the case when the machine configuration is the same at two limit ordinal stages \(\alpha<\beta\) and no cell that is showing 0 at stage \(\alpha\) ever
shows a 1 between these two stages. This configuration will then repeat (at least) \(\omega\) many times and the last clause ensures that the configuration at the limit of
these repeats will again be the same. The machine then has no alternative but to continue its sisyphean task. 

With the appropriate modifications the same description of repeating also works in the case of CRITTMs. Specifically, we say that a CRITTM computation repeats if the machine
configuration is the same at two cardinal stages \(\kappa<\lambda\) and no cell that is showing 0 at stage \(\kappa\) ever shows a 1 between these two stages. Note that we do not
require the stages \(\kappa\) and \(\lambda\) to be limit cardinals. This is not needed since the same argument as before shows that the machine cannot escape this
strong repeating pattern: the once repeated configuration will repeat again \(\omega\) many times\footnote{Note that the length of time between repeats gets progressively longer.
It will follow from later results that this additional time cannot be used in any meaningful way.} 
and the limit configuration will be the same.

\begin{proposition}
\label{prop:repeating time}
Every CRITTM computation halts or begins repeating before time \(\aleph_{\omega_1}\).
\end{proposition}

\begin{proof}
The proof is very similar to the one in the ITTM case. Assume that the machine hasn't halted by time \(\aleph_{\omega_1}\). By a cofinality argument there is a countable
ordinal \(\alpha_0\) such that all of the cells that will have stabilized before time \(\aleph_{\omega_1}\) have already done so by time \(\aleph_{\alpha_0}\). Similarly, there is a
countable ordinal \(\alpha_1>\alpha_0\) such that each of the nonstabilized cells will have flipped its value at least once between \(\aleph_{\alpha_0}\) and \(\aleph_{\alpha_1}\).
Construct the increasing sequence of ordinals \((\alpha_n)_{n<\omega}\) in this way and let 
\(\alpha_\omega=\sup_n\alpha_n<\omega_1\). Consider the configuration of the machine at time 
\(\aleph_{\alpha_\omega}\). The stabilized cells are showing their values and the nonstabilizing cells have flipped unboundedly many times below \(\aleph_{\alpha_\omega}\) and
are thus showing 1. This is the same configuration as at time \(\aleph_{\omega_1}\). If we now repeat the construction starting with some \(\alpha_0'\) above \(\alpha_\omega\),
we find the same configuration at some time \(\aleph_{\alpha_\omega'}\). Then, by construction, the computation between \(\aleph_{\alpha_\omega}\) and \(\aleph_{\alpha_\omega'}\)
repeats.
\qed
\end{proof}

We conclude this section by determining the precise computational power of CRITTMs. Recall that we introduced them as an attempt to define a minimal natural strengthening of the ITTM
model which could decide the strong ITTM halting problem without explicitly adding a halting oracle. It turns out that we hit our mark perfectly as the following theorem shows.

\begin{theorem}
CRITTMs compute the same functions as ITTMs with a \(0^\blacktriangledown\) oracle. In this sense the two models are computationally equivalent.\footnote{This theorem gives
an alternative proof (and was originally conceived by observing) that the sets \(0^\blacktriangledown\) and \(S\) are infinite time Turing equivalent; the strong halting oracle
allows us to foresee repeating patterns.}
\end{theorem} 

\begin{proof}
We have already seen in Proposition \ref{prop:halting decidable} that \(0^\blacktriangledown\) is decidable by a CRITTM, so any computation of an ITTM with a \(0^\blacktriangledown\) oracle can be simulated
on a CRITTM. We must now show that the converse holds.

Consider a CRITTM computation \(\varphi_p(x)\). By Proposition \ref{prop:repeating time} this computation halts or begins repeating in the CRITTM sense by time \(\aleph_{\omega_1}\). However, in each time
interval \([\aleph_\alpha,\aleph_{\alpha+1})\) the computation is in effect an ITTM computation and thus begins repeating in the ITTM sense by time \(\aleph_\alpha+\beta_\alpha\)
for some countable \(\beta_\alpha\). We see that the machine performs no useful computations for longer and longer periods of time. It is this behaviour that allows us to compress
the CRITTM computation. Given a CRITTM program \(p\), we let \(\tilde{p}\) be the program \(p\) with all mention of the \textit{cardinal} state omitted (transforming it into an
ITTM program). We then define \(p_*\) as the following ITTM program: given an input \(z\), decode it into the contents of the three machine tapes, the machine state and the head
position and run \(\tilde{p}\) on this decoded configuration.
By setting aside space to store information on which cells have changed
value since the start of the simulation and at each step comparing the current simulated
configuration to the one coded by \(z\), the program \(p_*\) can halt if the simulation of
\(\tilde{p}\) repeats the configuration coded by \(z\) in the ITTM sense.

Consider the following program for an ITTM with a \(0^\blacktriangledown\) oracle which simulates the CRITTM computation \(\varphi_p(x)\). First, construct the program \(p_*\). At each
subsequent step, code the simulated configuration into a real \(z\) and query the oracle whether \(\varphi_{p_*}(z)\) halts. If not, simulate a step of \(p\). Otherwise, 
\(\varphi_p(x)\) has reached its ITTM repeating configuration between cardinal stages, which must also be the configuration at the next cardinal time. We can therefore jump ahead in our
simulation, setting its state to \textit{cardinal} and moving the head back to the beginning of the tape, after which we continue with our procedure. Clearly, if \(\varphi_p(x)\)
halts, this simulation will halt with the same output.
\qed
\end{proof}

\section{Clockable and Writable Ordinals}

In this section we wish to examine the various classes of ordinals connected with CRITTM computation. In particular, we are interested in CRITTM-clockable
and (eventually/accidentally) CRITTM-writable ordinals, where these concepts are defined analogously to the ITTM case. In the sequel, when we write clockable/writable, we
always mean ITTM-clockable/writable. 

Let us first present some examples regarding CRITTM-clockable ordinals.
\begin{itemize}
\item The countable CRITTM-clockable ordinals are precisely the clockable ordinals. Clearly all clockable ordinals are CRITTM-clockable. In fact, since we stipulated that the 
\textit{cardinal} state is only used at uncountable stages, the very same program that clocks \(\alpha\) on an ITTM can be used to clock it on a CRITTM. Furthermore, no new
clockable ordinals appear since at countable stages CRITTMs behave no differently to ITTMs.

%
%

\item If \(\aleph_\alpha\) and \(\aleph_\beta\) are CRITTM-clockable then so are \(\aleph_{\alpha+\beta}\) and \(\aleph_{\alpha\cdot\beta}\). We only give a sketch of the algorithm
for the first part, leaving the second to the reader. If \(\aleph_\alpha\) and \(\aleph_\beta\) are CRITTM-clockable, we can modify the algorithms clocking them to only use the even- 
and odd-numbered cells on the tape, respectively, without changing the running time. We then design an algorithm which uses these modified programs to first clock \(\aleph_\alpha\)
and then run the program for clocking \(\aleph_\beta\), keeping track of which part of the computation we are executing by means of a master flag. This clocks 
\(\aleph_{\alpha+\beta}\).

\end{itemize}

There is a very important observation to be made regarding the last bullet point above. The observant reader will have noticed the awkward wording in the algorithm given; we are 
referring to the phrase ``run the program for clocking \(\aleph_\beta\)''. Why not simply say ``clock \(\aleph_\beta\)''? We feel that the issue is most easily seen through an example: 
knowing the algorithm for clocking \(\omega_1\), can we devise an algorithm for clocking \(\omega_1+\omega_1\)? Surely we can. Just run two copies of the program clocking \(\omega_1\)
one after another. But does this actually clock \(\omega_1+\omega_1\)? Let us take a closer look and fix the program clocking \(\omega_1\) to be the program which waits for the 
first \textit{cardinal} state that appears and then immediately halts. Consider what happens when we run two copies of this program in succession. The first copy waits for the
\textit{cardinal} state, which appears at time \(\omega_1\), and passes control to the second copy. The second copy then waits for the \textit{cardinal} state. However, this
state doesn't occur at time \(\omega_1+\omega_1\), but only at \(\omega_2\). 
We remark that this
phenomenon doesn't occur in ITTM computation due to a certain homogeneity the class of cardinals lacks; the next limit ordinal above an ordinal \(\alpha\) is always
\(\alpha+\omega\), but the next cardinal above \(\alpha\) has no such uniform description. It could, perhaps, be argued that this issue makes the notion of the length
of a CRITTM computation meaningless, as this quantity changes based on the time at which we run the computation.

The discussion in the previous paragraph leads us to formulate the following theorem, which places strong restrictions on decompositions of CRITTM-clockable ordinals.

\begin{theorem}
\label{thm:fast clock}
Let \(\alpha\) and \(\beta\) be ordinals with \(|\alpha|\geq|\beta|\). 
If \(\alpha+\beta\) is CRITTM-clockable then \(\beta\) is countable.
\end{theorem}

\begin{proof}
The assumption \(|\alpha|\geq|\beta|\) may be restated as \(|\alpha+\beta|=|\alpha|\). Therefore the last cardinal stage passed in a CRITTM computation of length \(\alpha+\beta\)
is \(|\alpha|\). In particular, the \textit{cardinal} state doesn't appear during the last \(\beta\) steps of computation.

Let \(p\) be the program clocking \(\alpha+\beta\) and let \(x_\alpha\) be (a real coding) the content of the machine tapes after \(\alpha\) many steps of computation.
Since no \textit{cardinal} states appear during the last \(\beta\) many steps of the computation, we are, in effect, performing an ITTM computation with
input \(x_\alpha\) which halts after \(\beta\) many steps. As we know the halting times of ITTM computations to be countable, \(\beta\) must be countable.
\qed
\end{proof}

\begin{corollary}
CRITTM-clockable ordinals are not closed under ordinal arithmetic.
\end{corollary}

\begin{proof}
Consider, as before, \(\omega_1+\omega_1\). Theorem \ref{thm:fast clock} implies that this ordinal is not CRITTM-clockable.
\qed
\end{proof}

Theorem \ref{thm:fast clock} raises some interesting questions. For example, given \(\alpha\), which countable `tails' \(\beta\) give CRITTM-clockable 
ordinals? If \(\alpha\) is CRITTM-clockable we might be inclined to say that precisely the clockable \(\beta\) arise, but this is incorrect. Keep in mind that the first part of
the computation might have produced some useful output for us to utilize during the last \(\beta\) steps. In particular, assuming \(\alpha\) is uncountable, any writable
real may be produced to serve as input to the subsequent computation. The following proposition illustrates this fact.

\begin{proposition}
The ordinal \(\omega_1+\Sigma+\omega\) is CRITTM-clockable.\footnote{This example is due to Joel David Hamkins.}
\end{proposition}

\begin{proof}
In \cite{welch-ittmordinals:2000} it is shown that the configuration of the universal ITTM\footnote{The
universal ITTM is the machine that dovetails the simulation of all ITTM programs on empty
input. It can be used to produce a stream consisting of all accidentally writable reals.} 
at time \(\zeta\) is its repeating configuration and that it repeats for the first time at time \(\Sigma\).
This configuration also appears at time \(\omega_1\). Let our CRITTM simulate the universal ITTM until time \(\omega_1\) at which point it stores the simulated configuration and
starts a new simulation of the universal ITTM. The machine also checks at limit steps whether the current simulated configuration is the same as the stored configuration and
whether this has happened for the second time. This will occur at time \(\omega_1+\Sigma\), our machine will detect this and halt at time \(\omega_1+\Sigma+\omega\).
\qed
\end{proof}

Another question we might ask is how CRITTM-clockability propagates. It is easily shown that \(\alpha^+\), the cardinal successor of \(\alpha\), is CRITTM-clockable if \(\alpha\) is.
Does it also inherit to cardinals in reverse? That is to say, if \(\alpha\) is CRITTM-clockable, must \(|\alpha|\) be as well? The following theorem shows that this fails in the
most spectacular way possible.

\begin{theorem}
There is a CRITTM-nonclockable cardinal \(\kappa\) such that the ordinal \(\kappa+\omega\) is CRITTM-clockable.
\end{theorem}


\begin{proof}
Let \(\kappa\) be the first CRITTM-nonclockable cardinal. To show that \(\kappa+\omega\) is CRITTM-clockable, dovetail the simulations of all CRITTM programs on input 0. We can
perform this in such a way that after \(\lambda\) many steps, for \(\lambda\) a cardinal, we have simulated exactly \(\lambda\) many steps of each computation. When a
\textit{cardinal} state is reached we use the next \(\omega\) many steps to check whether any of the simulated computations halt at the next step. If a computation halts we continue
with the simulation. Otherwise, we've happened upon the first CRITTM-nonclockable cardinal \(\kappa\) and we halt, having clocked \(\kappa+\omega\).
\qed
\end{proof}

The theorem shows that the fact that \(\kappa^+\) is CRITTM-clockable doesn't
imply that \(\kappa\) itself is CRITTM-clockable. The proof is merely the trivial observation that we have shown that \(\kappa+\omega\) is CRITTM-clockable, where \(\kappa\)
is the least CRITTM-nonclockable cardinal, hence \(\kappa^+\) is CRITTM-clockable.
This is in stark contrast to the speed-up theorem of \cite{hamkinslewis-ittm:2000}, which states that if
\(\alpha+n\) is clockable for some finite \(n\) then \(\alpha\) is clockable as well.

We now consider CRITTM-writable reals and ordinals and their more transient relatives. The same proofs as in the ITTM case show that the class of accidentally
CRITTM-writable reals properly contains the class of eventually CRITTM-writable reals which properly contains the class of CRITTM-writable reals and that the corresponding classes of ordinals
are downward closed. Also, no additional effort need go into proving that the classes of eventually CRITTM-writable and CRITTM-writable ordinals are closed under ordinal arithmetic.

In keeping with the notation from the ITTM model, we introduce the following ordinals
\begin{align*}
\lambda_C&=\sup\{\alpha;\alpha \text{ is CRITTM-writable}\}\\
\zeta_C&=\sup\{\alpha;\alpha \text{ is eventually CRITTM-writable}\}\\
\Sigma_C&=\sup\{\alpha;\alpha \text{ is accidentally CRITTM-writable}\}
\end{align*}

The same proof that shows \(\lambda<\zeta<\Sigma\) in the ITTM model can be used here to show that \(\lambda_C<\zeta_C<\Sigma_C\). We also have the following proposition.

\begin{proposition}
\(\Sigma_C\) is a countable ordinal.
\end{proposition}

\begin{proof}
Accidentally CRITTM-writable ordinals are countable by definition. We now prove that there are only countably many
accidentally CRITTM-writable reals, whence the proposition follows immediately.

Consider a CRITTM program \(p\). We have shown that \(p\) either halts or has repeated by time \(\aleph_\alpha\) for some countable \(\alpha\). 
Recall that the computation of \(p\) on input \(x\) either halts or repeats (in the ITTM sense) at least once by some countable time \(\delta_x\).
Let \(\beta<\alpha\). Within the time interval \([\aleph_\beta,\aleph_{\beta+1})\) the program \(p\) behaves just like an ITTM program and must
therefore halt or repeat by time \(\aleph_\beta+\delta_{x_\beta}\), where \(x_\beta\) codes the contents of the tapes at time \(\aleph_\beta\). This means that the program \(p\) 
in fact produces only countably many reals in the time interval \([\aleph_\beta,\aleph_{\beta+1})\). Repeating this argument for all \(\beta<\alpha\), we see that only countably many 
reals appear during the computation of \(p\) and since there are only countably many programs, there can only be countably many accidentally CRITTM-writable reals.
\qed
\end{proof}

But what is the relation between the ordinals \(\lambda_C,\zeta_C,\Sigma_C\) and their ITTM counterparts? Clearly the supremum of a given CRITTM class of ordinals is at least as big
as the supremum of the corresponding ITTM class, i.e.\ every (eventually/accidentally) writable ordinal is also (eventually/accidentally) CRITTM-writable, but can we say more?
It is in fact easy to see that every eventually writable real is CRITTM-writable: if a real \(x\) is eventually written by a program \(p\), we may use a CRITTM to simulate \(p\) for 
\(\omega_1\) many steps, at which point we halt, having written \(x\). Therefore \(\lambda_C\geq\zeta\) and we can go even further.

\begin{proposition}
\(\zeta\) is CRITTM-writable.
\end{proposition}

\begin{proof}
We begin by dividing the output tape into \(\omega\) many \(\omega\)-blocks. We now proceed to enumerate all ITTM programs. When a program \(p\) is enumerated, first determine
whether \(\varphi_p(0)\) stabilizes and, if it does and if its stabilized output codes an ordinal\footnote{Checking whether a real codes an ordinal can be done using the count-through
algorithm of \cite{hamkinslewis-ittm:2000}}, copy its output to the next empty \(\omega\) block on the output tape. Having
done this, resume enumerating programs. To deal with each particular program we need to pass only a single cardinal stage. Therefore we will have copied all eventually writable
ordinals onto the output tape by time \(\aleph_\omega\), which we can recognize. At this point we use \(\omega\) many steps to combine the codes on the output tape into
a code for the sum of all eventually writable ordinals and then halt.

This CRITTM algorithm writes an ordinal which is at least as big as \(\zeta\), therefore \(\zeta\) is CRITTM-writable.
\qed
\end{proof}

We have now seen that many new ordinals become writable and eventually writable when passing from the ITTM model to the CRITTM model. But what about accidentally writable ordinals?

\begin{question}
Is \(\Sigma=\Sigma_C\)?
\end{question}

While we currently do not have an answer to this question, let us present a brief heuristic justification for why we believe the proposed equality to be true. We have seen that \(\lambda<\lambda_C\) and \(\zeta<\zeta_C\), but these inequalities
are not particularly surprising given that the concepts of (eventual) writability are intimately connected with the computational power of the particular model under observation.
Accidental writability, however, seems to be a much more robust notion. 
It is not at all clear how any real could appear during a CRITTM computation and not during an ITTM computation
since, as we have seen, CRITTM computations are more or less just ITTM computations with some irrelevant padding inserted.

%
%

We now turn to the CRITTM counterpart of \(\gamma\), the supremum of the clockable ordinals:
\[\gamma_C=\sup\{\alpha;\alpha \text{ is CRITTM-clockable}\}\]
Welch shows in \cite{welch-ittmordinals:2000} that \(\gamma=\lambda\). We intend to prove a somewhat similar statement in the context of CRITTMs.

\begin{lemma}
If \(\alpha\) is clockable then \(\aleph_\alpha\) is CRITTM-clockable.
\end{lemma}

In the proof we shall introduce an algorithm, based on an argument of \cite{hamkinslewis-ittm:2000}, which we call
the \emph{cardinal step count-through algorithm} and which will be very useful in what
follows.

\begin{proof}
Since \(\alpha\) is clockable, we can
write a real coding it in a countable number of steps. We now perform the count-through algorithm\footnote{That is, we use the improved version of the algorithm, which counts through a code for \(\alpha\) in \(\alpha\) many steps.} 
of \cite{hamkinslewis-ittm:2000} on this code, but with a small modification: after each step of the algorithm our machine records the current head position and state in some way and
waits for the next \textit{cardinal} state, at which point it decodes the previous configuration and continues the algorithm. Of course, this modification requires keeping track
of limit cardinal stages, when all previous information should be discarded, but this can be dealt with using the same bookkeeping devices that are used to keep track of limit-of-limits
stages in ITTM computations. This CRITTM algorithm clocks \(\aleph_\alpha\).
\qed
\end{proof}

Based on this lemma we can state that \(\gamma_C\geq \aleph_\gamma\). While the possibility of having \(\gamma_C=\aleph_\gamma\) is certainly alluring, this equality does not
hold. We can easily use the cardinal step count-through algorithm to show that for each CRITTM-writable ordinal \(\alpha\) there exists a CRITTM-clockable ordinal at least as large as 
\(\aleph_\alpha\). Therefore we must have \(\gamma_C\geq \aleph_{\lambda_C} > \aleph_\gamma\).

%
%

Since there is a CRITTM-clockable cardinal above any CRITTM-clockable ordinal (the cardinal
successor works), \(\gamma_C\) must be a cardinal. We now proceed to determine exactly which cardinal it is.

\begin{lemma}
\label{lemma:gen welch}
Every CRITTM computation with input 0 repeats its \(\aleph_{\zeta_C}\) configuration at \(\aleph_{\Sigma_C}\).
\end{lemma}

In the interest of brevity we omit the slightly technical proof. Let us just remark that
the proof is a straightforward generalization of the arguments in \cite{welch-ittmordinals:2000} to the
present context.

\begin{theorem}
\(\gamma_C=\aleph_{\lambda_C}\)
\end{theorem}

\begin{proof}
We have already seen that \(\gamma_C\geq\aleph_{\lambda_C}\).
Lemma \ref{lemma:gen welch} implies that every halting CRITTM computation with input 0 must halt before time \(\aleph_{\zeta_C}\). 
We therefore have \(\gamma_C\leq \aleph_{\zeta_C}\).
Consider a halting CRITTM program \(p\). Begin enumerating accidentally CRITTM-writable ordinals \(\alpha\). For each \(\alpha\) use the cardinal step count-through algorithm
to simulate the computation of \(\varphi_p(0)\) up to \(\aleph_\alpha\), while simultaneously keeping track of the deleted initial segment \(\beta\) of \(\alpha\). 
Eventually an \(\alpha\) will be enumerated which is large enough that the simulation halts before time \(\aleph_\alpha\). When this happens, we halt with output \(\beta+1\). 
Therefore the program \(p\) halts by time \(\aleph_{\beta'}\) for some writable \(\beta'\), which means that \(\gamma_C\leq \aleph_{\lambda_C}\). Putting this together, we have shown
that \(\gamma_C=\aleph_{\lambda_C}\).
\qed
\end{proof}

Upon reflection, the properties of CRITTMs explored in this paper require very little of
the structure of the class of cardinals. One could equally well have considered
\emph{\(A\)-recognizing} ITTMs for some club class of limit ordinals \(A\). For example,
we could have considered ITTMs with a state recognizing ordinal multiples of \(\Sigma\).
Some of our results would generalize, but note that we have often used the result
that any ITTM computation halts or repeats before \(\aleph_1\) and this property holds
with \(\Sigma\) only for ITTM computations with empty input, so we cannot expect to get
an equivalent model.

The generalization to arbitrary \(A\) leads to some familiar results (e.g.\ any computation in this model either halts or begins repeating before the \(\omega_1\)-st element of \(A\)), but
it also has serious issues. Since the class \(A\) may lack the uniformity of the classes
of limit ordinals or cardinals, even the existence of a universal machine is not clear (if
there is a change in the frequency of the \(A\)-stages and the machine cannot anticipate this,
a running simulation may not produce correct results).

\begin{question}
What properties should the club class \(A\) have to get a meaningful notion of \(A\)-recognizing ITTMs?
\end{question}

Instead of studying them in isolation, we can also compare the \(A\)-recognizing
ITTM models for various \(A\). This leads to the concept of reductions between them. 
For example, any \(A\)-recognizing ITTM can
simulate a limit ordinal recognizing ITTM (which is just an ordinary ITTM).
This endows the collection of club classes of ordinals with a reducibility degree structure, 
similar to but distinct from the usual Turing degrees.

\begin{question}
What are the degrees thus obtained?
\end{question}

\subsubsection{Acknowledgements.}
The author wishes to thank Joel David Hamkins for the many fruitful discussions and
his suggestions of improvements to the paper.

The author's work was supported in part by the Slovene Human Resources and Scholarship Fund.

\bibliographystyle{splncs}

\end{document}